\documentclass[12pt,oneside]{amsart}

\usepackage[utf8]{inputenc}
\usepackage{bm, verbatim, graphicx, subcaption, amssymb, color, mathabx, mathtools, enumitem, mathrsfs, amsmath, amsthm, ulem, xspace,tabularx }
\usepackage{xcolor}
\usepackage{subcaption}
\usepackage[pdftex,bookmarks,colorlinks,breaklinks]{hyperref}
\hypersetup{linkcolor=blue,citecolor=red,filecolor=dullmagenta,urlcolor=darkblue}
\usepackage[letterpaper, portrait, lmargin=.9in, rmargin=.9in, tmargin=.8in, bmargin=.9in, headsep=0.25in]{geometry}
\newtheorem{theorem}{Theorem}[section]
\newtheorem{lemma}[theorem]{Lemma}
\newtheorem{question}{Question}[section]
\newtheorem{claim}{Claim}[section]

\newcommand\roundup[1]{\left\lceil#1\right\rceil}
\newcommand\rounddown[1]{\left\lfloor#1\right\rfloor}

\usepackage{etoolbox}

    \makeatletter
    \patchcmd{\@setauthors}{\MakeUppercase}{}{}{}
    \makeatother\usepackage{etoolbox}

    \makeatletter
    \patchcmd{\@setauthors}{\MakeUppercase}{}{}{}
    \makeatother

      \makeatletter
      \def\@setcopyright{}
      \def\serieslogo@{}
      \makeatother

\begin{document}
    \author{Guantao Chen$^1$}
    \author{Alireza Fiujlaali$^1$}
    \author{Anna Johnsen-Yu$^{1,2}$}
    \author{Jessica McDonald$^3$}
    \address{$^1$Department of Mathematics and Statistics, Georgia State University, Atlanta, GA USA 30303-2918. Supported in part  NSF grant DMS-2154331.}
\address{$^{2}$Department of Mathematics,  Vanderbilt University, Nashville, TN 37240-0001 }
    \address{$^3$Department of Mathematics and Statistics,  Auburn University, Auburn, AL 36849. Supported in part by Simons Foundation Grant \#845698 and NSF grant DMS-2452103.}

\title[]{On graphs with girth at least five achieving Steffen's edge coloring bound}

\begin{abstract}
Vizing and Gupta showed that the chromatic index $\chi'(G)$ of a graph $G$ is bounded above by $\Delta(G) + \mu(G)$, where $\Delta(G)$ and $\mu(G)$ denote the maximum degree and the maximum multiplicity of $G$, respectively. Steffen refined this bound, proving that $\chi'(G) \leq \Delta(G) + \left\lceil \mu(G)/\left\lfloor g(G)/2 \right\rfloor \right\rceil$, where $g(G)$ is the girth of the graph $G$. A {\it ring graph} is a graph obtained from a cycle by duplicating some edges. The equality in Steffen's bound is achieved by ring graphs of the form $\mu C_g$, obtained from an odd cycle $C_g$ by duplicating each edge $\mu$ times. We answer two questions posed by Stiebitz et al. regarding the characterization of graphs which achieve Steffen's bound. In particular, we show that if $G$ is a critical graph which achieves Steffen's bound with $g(G)\geq 5$ and $\chi'(G)\geq \Delta+2$, then $G$ must be a ring graph of odd girth.
\end{abstract}

   \subjclass[2010]{05C15}
   \keywords{Edge Coloring, Chromatic Index, Ring Graphs, Graph Classification}

   \date{\today}

   \maketitle
\section{Introduction}

We generally follow the notation and definitions of West \cite{MR1367739}.  Graphs under consideration may have multiple edges, but do not have loops.
A {\it $k$-edge-coloring} of a graph $G$ is a labeling $f: E(G)\to [k]:=\{1, \dots, k\}$, such that the edges incident to any single vertex are labeled with distinct colors. The {\it chromatic index} of a graph $G$, denoted by $\chi'(G)$, is the least $k$ such that $G$ has a $k$-edge-coloring.

Holyer~\cite{MR635430} proved that
determining the chromatic index of an arbitrary graph is NP-complete, even when restricted to cubic simple graphs.
Since all edges incident with a single vertex must be labeled with a different color, $\chi'(G)\geq\Delta(G)$, where $\Delta(G)$ denotes the maximum degree of $G$. In 1949,
Shannon~\cite{MR30203}  showed that $\chi'(G)\leq (3/2)\Delta(G)$. In the 1960s, Vizing~\cite{MR240000} and Gupta~\cite{Gupta} independently established another upper bound involving the {\it multiplicity} of a graph $G$, denoted by $\mu(G)$; that is, the maximum number of parallel edges sharing two common vertices. In particular, they proved that $\chi'(G)\leq \Delta(G) + \mu(G)$. This is commonly referred to as ``Vizing's theorem'' since Gupta's proof was discovered after Vizing's result had become widely known. This result is particularly appealing when applied to  simple graphs since it reduces the possible values of $\chi'(G)$ to two options: $\Delta(G)$ and $\Delta(G)+1$.

The {\it girth} of $G$, denoted by $g(G)$,  is the smallest integer $g \ge 3$ such that $G$ has a cycle length $g$ if such a cycle exists and $g(G) = \infty$ otherwise. By taking into account the girth of a graph, Steffen refined Vizing's theorem as follows.
\begin{theorem}[Steffen~\cite{MR1754344}]\label{Thm-Steffen}
For any graph $G$,
    $ \label{eqn:Steffen}
        \chi'(G)\leq \Delta(G)+\roundup{\dfrac{\mu(G)}{\rounddown{g(G)/2}}}$.

\end{theorem}

Let $st(\Delta, \mu, g) =\max\{\chi'(G) :  \mbox{  $G$ is a graph with } \Delta(G) = \Delta,\  \mu(G) = \mu,\  g(G) = g\}$.
Steffen's theorem shows that $st(\Delta, \mu, g) \le \Delta +\roundup{\mu/\rounddown{g/2}}$.
Stiebitz, Scheide, Toft, and Favrholdt ~\cite[ch. 9, pp. 249-250]{MR2975974}  raised the following two questions. Here, a {\it ring graph} is obtained from a cycle by duplicating some edges.

\begin{question}
\label{GuptaConj1}
    For which triples $(\Delta,\mu,g)$ is   $st(\Delta, \mu, g)  = \Delta+\roundup{\mu/\rounddown{g/2}}$?
\end{question}
\begin{question}
\label{GuptaConj2}
    Let $g$ be an odd integer and let $\mu$ be an integer with $\mu\geq g\geq 3$. Is there a graph $G$ with $g(G)=g$ and $\mu(G)=\mu$ achieving Steffen's bound, but not containing a ring graph of order $g$ with the same chromatic index as $G$?
\end{question}

Stiebitz, Scheide, Toft, and Favrholdt~\cite{MR2975974} observed that for every odd integer \(g \ge 3\), the multigraph
\(G = \mu C_g\),
obtained from the odd cycle \(C_g\) by duplicating each edge \(\mu\) times, satisfies
\[
\chi'(G) = \Delta(G) + \Bigl\lceil \frac{\mu}{\lfloor g/2\rfloor}\Bigr\rceil.
\]
Hence, for any positive integer \(\mu\) and odd \(g \ge 3\), the triple \((2\mu, \mu, g)\) is a solution to Question~\ref{GuptaConj1}.
When \(g \ge 4\) is even, the disjoint union of \(\mu C_g\) and \(\mu C_{g+1}\) also attains this bound; in that case the triple \((2\mu,\mu,g)\) again satisfies Question~\ref{GuptaConj1}.
Thus, one obtains
\[
\mathrm{st}(2\mu,\mu,g) = 2\mu + \Bigl\lceil \frac{\mu}{\lfloor g/2\rfloor}\Bigr\rceil
\]
for all integers \(\mu\ge1\) and \(g\ge3\).

Our main result is Theorem \ref{thm-main} below, which provides a negative answer to Question~\ref{GuptaConj2} for $g\geq 5$ and $\mu(G) \geq \lfloor g/2 \rfloor+1$, which is slightly more broad than the original question. Note that for a graph $G$ achieving Steffen's bound, this restriction on $\mu(G)$ is equivalent to requiring $\chi'(G)\geq \Delta(G)+2$.

\begin{theorem}\label{thm-main}
Let $g\ge 5$ be an integer and let $G$ be a graph with girth $g(G) \geq g$ and  $\mu(G) \geq \lfloor g/2 \rfloor+1$. If $\chi'(G) = \Delta(G) + \lceil \mu(G)/\lfloor g/2 \rfloor\rceil$, then $G$ contains a ring graph with the same chromatic index as $G$.
\end{theorem}

Our proof technique for Theorem \ref{thm-main} unfortunately does not apply to graphs $G$ with $g(G) =4$, so some new ideas would be needed to attack this case.  When $g(G)=3$ however, there is a positive answer to Question~\ref{GuptaConj2}. To see this, let $\mu K_n$ be obtained from the complete graph on $n$ vertices by duplicating each edge $\mu$ times. If $n\geq3$ is odd, then $\chi'(\mu K_n) = \mu n= \Delta(\mu K_n) + \mu(\mu K_n)$,  and $((n-1)\mu, \mu, 3)$ is a triple satisfying Question~\ref{GuptaConj1}. Since $\mu K_n$ does not contain a ring graph with the same chromatic index when $n>3$, this answers Question~\ref{GuptaConj2}  positively.

A graph is  (edge-coloring) \textit{critical} if $\chi'(H) < \chi'(G)$ for any proper subgraph $H \subset G$. Theorem~\ref{thm-main} implies that if $G$ is a critical graph with $g(G)\geq 5$ and $\mu(G) \ge \lfloor g(G)/2 \rfloor +1$ which achieves Steffen's bound, then $G$ is a ring graph, and moreover $g(G)$ must be odd. The oddness comes from the fact that a ring graph with an even number of vertices is a bipartite graph (and so its chromatic index equals its maximum degree).

The remainder of the paper is organized as follows. In Section \ref{prelim}, we introduce more notation and prove some preliminary lemmas. The proof of Theorem \ref{thm-main} for $g\geq 6$ is given in Section \ref{mainpf}.
The proof of Theorem \ref{thm-main} for $g=5$ is given separately, in Section \ref{oddipf}.

Throughout the paper, we use the following notation. If the graph \(G\) is clear from the context, we omit \(G\) from graph parameters; in particular, we write \(\chi'\), \(\Delta\), and \(\mu\) for \(\chi'(G)\), \(\Delta(G)\), and \(\mu(G)\), respectively.
We denote that two vertices \(u\) and \(v\) are {\it adjacent} in \(G\) by \(u \sim v\).
We let \(|G| = |V(G)|\).
For subgraphs \(F, H \subseteq G\), let
\(N_H(F) = N(V(F)) \cap V(H)\) be the set of neighbors of \(V(F)\) in \(H\),
and let \(d_H(v)\) denote the degree of \(v\) in \(H\).

\section{Preliminaries}\label{prelim}

The \textit{density} of a graph $G$ is defined  as follows:
\[\Gamma(G) = \max\left\{\roundup{\frac{2|E(H)|}{|H|-1}}\ : \ H\subseteq G, |H| \text{ odd}\right\}.\] Since each matching in a subgraph $H$ with an odd number of vertices has at most $(|H|-1)/2$ edges, it follows that  $\chi'(G)\geq \Gamma(G)$ for all graphs $G$. Hence, $\Gamma(G)$ provides another lower bound for $\chi'(G)$ besides the maximum degree $\Delta(G)$.
The well-known Goldberg-Seymour Conjecture states that if $\chi'(G)\ge \Delta(G) +2$ then $\chi'(G) =\Gamma(G)$, which was confirmed by Chen, Hao, Jing, Yu and Zang in three different versions of papers~\cite{chen2025proof,ChenHaoYuZang2024ShortGS,Jing2024EdgeColoring}. We refer to  \cite{MR2975974} for literature on this conjecture.

\begin{theorem}[Chen, Hao, Jing, Yu, Zang]\label{Thm-CJZ}
For any graph $G$, if $\chi'(G) \ge \Delta(G) +2$ then $\chi'(G) = \Gamma(G)$. \end{theorem}

Theorem~\ref{Thm-CJZ} implies the following two results, each of which provides a lower bound on minimum degree.

\begin{lemma} \label{critical}
    Let $G$ be a critical graph on $n$ vertices. If  $\chi'(G) \geq \Delta + 2$,
    then $n$ is odd and for every vertex $v\in V(G)$,
    $$d_{G}(v) = \sum_{w\ne v} (\chi' -1 -d_G(w)) +2  \ge (n-1)(\chi' -1 -\Delta) +2.$$
Moreover, if $n \ge g$ and $\chi'(G) = \Delta + \lceil \mu/\lfloor g/2 \rfloor\rceil \ge \Delta +2$ for some positive integer $g \ge 5$, then
    \( \delta(G) \ge  n\mu/g +1\), and  equality holds only if $\mu =g$.
    \end{lemma}

\begin{proof} By definition, $\Gamma(G) \le \chi'(G)$. Applying  Theorem~\ref{Thm-CJZ}, we find that $\Gamma(G)=\chi'(G)$ as $\chi'(G) \ge \Delta +2$. Since $G$ is critical, for every proper subgraph $H$ of $G$ with an odd number of vertices, we have that $\roundup{2|E(H)|/(|H|-1)} \le \Gamma(H) \le \chi'(H) < \chi'(G) = \Gamma(G)$.  Hence, $G$ itself is the only subgraph $H$ of $G$ with an odd number of vertices such that $\roundup{2|E(H)|/(|H|-1)} = \Gamma(G)$. Consequently, $n$ is odd, and for any edge $e\in E(G)$,
\( \chi'(G-e) =\chi'(G) -1\). Hence,
\[
 \roundup{\frac{|E(G-e)|}{(n-1)/2}} \le \chi'(G-e)= \chi'(G) -1 = \roundup{\frac{|E(G)|}{(n-1)/2}} -1.
\]
Since $\frac{|E(G-e)|}{\rounddown{(n-1)/2}} = \frac{|E(G)|-1}{(n-1)/2} \ge \roundup{\frac{|E(G)|}{(n-1)/2}} -1$, it follows that  $\frac{|E(G-e)|}{\rounddown{(n-1)/2}}$ is an integer and is equal to $\chi'(G-e)$. Hence, for an edge-coloring of $G-e$ with $\chi'(G-e)$ colors, each color class contains exactly $(n-1)/2$ edges. Therefore, \(E(G - e)\) can be partitioned into \(\chi'(G) - 1\) near-perfect matchings, each of which misses exactly one vertex of \(G\).

Let $v$ be the  vertex specified in Lemma~\ref{critical}, and  let $e$ be an edge not incident with $v$.  For each vertex $w\ne v$, if $w$ is not incident to $e$, there are exactly $\chi'(G) -1 -d_G(w)$ of these near-perfect matchings not containing the vertex $w$;  if $w$ is incident to $e$, there are exactly $\chi'(G) -1 -(d_G(w) -1)$ of these matchings  not containing the vertex $w$. Note that every one of these matchings contains $v$. Hence, $d_G(v)  = \sum_{w\ne v} (\chi'(G) -1 -d_G(w)) +2$. As $d_G(w) \le \Delta(G)$, we have the following.
\[
\delta(G) \ge (n-1)(\chi'(G) -1 -\Delta(G)) +2.
\]

Now suppose $\chi'(G)=\Delta + \roundup{\mu/\rounddown{g/2}}\geq \Delta + 2$ for some positive integer $g\geq 5$. Then by substituting $\Delta + \roundup{\mu/\rounddown{g/2}}$ for $\chi'(G)$, we obtain the following inequality.
\[
\delta(G) \ge (n-1)\left(\roundup{\frac{\mu}{\rounddown{g/2}}}-1\right) +2\ge (n-1)\left(\roundup{{\frac{2\mu}g}}-1\right) +2.
\]

Note that
\[
(n-1)\left(\roundup{\frac{\mu}{\rounddown{g/2}}}-1\right) +2 -  \left(\frac {n\mu}g +1\right)
\ge (n-1)\left(\roundup{\frac{2\mu}g}-\frac {\mu}g -1\right) -\frac {\mu}g +1.
\]
If $\mu <   g$, then $\frac {\mu +1}g \le 1$, and hence
$\roundup{\frac{2\mu}g} \ge 2 \ge 1+ \frac {\mu +1}g$. Thus,
\[
(n-1)\left(\roundup{\frac{2\mu}g}-\frac {\mu}g -1\right) -\frac {\mu}g +1 \ge \frac 1g ((n-1) -\mu) +1 \ge \frac 1g((n-1) -g) +1> 0.
\]

If $\mu =g$, then
$
(n-1)\left(\roundup{\frac{2\mu}g}-\frac {\mu}g -1\right) -\frac {\mu}g +1= 0$, as desired.

We now turn to the case $\mu >g$. In this case, $\roundup{2\mu/g} \ge 3$. If $\mu/g \le 1.5$,  then $4\mu/(3g) \le 2$, and hence $\roundup{2\mu/g} =3\ge 1 +4\mu/(3g)$.  If $\mu/g > 1.5$, then $(2/3) (\mu/g) \ge 1$, which implies that $\mu/g > 1 + \mu/(3g)$, and so $\roundup{2\mu/g} \ge \mu/g +\mu/g \ge 1+\mu/3g + \mu/g = 1 +4\mu/(3g)$.
Hence,
\[
(n-1)\left(\roundup{\frac{2\mu}g}-\frac {\mu}g -1\right) -\frac {\mu}g +1\ge \frac{(n-1)\mu}{3g} -\frac{\mu}g +1 \ge \frac {4\mu} g -\frac {\mu}g +1  > 0,
\]
as $n \ge 5$ and $\mu/g \ge 1$.
\end{proof}

We assume that each cycle $C$ has an orientation. Given two vertices $u, v\in V(C)$, let $C[u,v]$ denote the path connecting $u$ and $v$ along
the cycle according to its orientation  and let $C[v,u]$ denote the path connecting $u$ and $v$ along the cycle in the opposite direction. Note that $C=C[u,v]C[v,u]$. Similarly, for each path $P$, we 
denote by $P[u,v]$ the path beginning from the vertex $u$ and ending at the vertex $v$ along 
$P$. For any graph $G$, let $\underline{G}$ denote the underlying simple graph $G$, i.e., $\underline{G}$ is a simple graph with $V(\underline{G}) = V(G)$ such that  two vertices are adjacent in $\underline{G}$ if and only if they are adjacent in $G$.

\begin{lemma}\label{lem-shortcycle}
 Let $G$ be a graph and $C$ be a cycle of $G$ of length $g(G)$.
 Then the following hold.
 \begin{enumerate}
     \item If $|C| \ge 5$, then $|N_C(v)| \le 1$ for every $v\in V(G) \setminus V(C)$.
     \item If $|C| \ge 7$, then $|N_C(u)| + |N_C(v)| \le 1$ for every two adjacent vertices $u, v\in V(G)\setminus V(C)$.
     \item If $|C| \ge 8$, then for every path $v_1v_2v_3v_4v_5$ in $G-V(C)$,
     $\sum_{i=1}^5 |N_C(v_i)| \le 2$.
     \item If $|C| \ge 6$, then for every path $v_1v_2v_3$ in $G-V(C)$, $\sum_{i=1}^3 |N_C(v_i)| \le 2$.
 \end{enumerate}
\end{lemma}

\begin{proof}
Let $C$ be a cycle of $G$ with $|C|=g(G)$.

To prove (1), suppose $|C|\geq 5$ and assume to the contrary there is a vertex $v\in V(G)\backslash V(C)$ such that $N_C(v)$ contians two distinct vertices $x$ and $y$. Assume without loss of generality that $|E(C[x,y])|\leq |C|/2$. Then the cycle $vC[x,y]v$ has $2 + |C|/2<|C|$ edges, a contradiction to our assumption that $|C|\geq 5$.

To prove (2), suppose $|C|\geq 7$ and assume to the contrary that there are two adjacent vertices $u,v\in V(G)\backslash V(C)$ such that $|N_C(u)|+|N_C(v)|\geq 2$.
By (1), we may assume that there are two vertices $x\in N_C(u)$ and $y\in N_C(v)$. If $x=y$, then $\underline{G}$ has a cycle of length $3$, a contradiction.  Thus, $x\ne y$. Assume without loss of generality that $|E(C[x,y])| \le |C|/2$. Then $uC[x,y]vu$ is a cycle of $\underline{G}$ with $3 +|C|/2 < |C| $ edges, a contradiction to our assumption that $|C|\geq 7$.

To prove (3), suppose $|C|\geq 8$ and assume to the contrary that there is a path $v_1v_2v_3v_4v_5$ of vertices in $V(G)\backslash V(C)$ such that $\sum_{i=1}^5 |N_C(v_i)| \ge 3$.
By (1), there are three vertices in $\{v_1, v_2, v_3, v_4, v_5\}$ which are adjacent to $C$, and by (2) these three vertices must be independent. Thus, we may assume there are three distinct vertices $x_1, x_3, x_5\in V(C)$ such that $v_ix_i\in E(G)$ for $i\in\{1, 3, 5\}$.  Assume without loss of generality that $x_1$, $x_3$, and $x_5$ are labeled in order along a fixed orientation of $C$. Then there are three cycles $D_1, D_2, D_3$ in $\underline{G}$ defined by $D_1:= v_1C[x_1,x_3]v_3v_2v_1$,  $D_2:= v_3C[x_3,x_5]v_5v_4v_3$, and $D_3:= v_5C[x_5,x_1]v_1v_2v_3v_4v_5$ (see Figure \ref{fig:2-3(3)}). Since $|C|$ is a shortest cycle in $G$, we have that
    \[
3|C| \le \sum_{i=1}^3 |E(D_i)| = |E(C)| +4 + 4 +6  = |C| +14.
    \]
Hence, $|C| \le 7$, a contradiction.

The proof of (4) is similar to the proof of (3) with $v_1$, $v_2$, and $v_3$ in place of $v_1$, $v_3$, and $v_5$ and $x_1$, $x_2$, and $x_3$ in place of $x_1$, $x_3$, and $x_5$ (see Figure \ref{fig:2-3(4)}). We find three cycles $D_1$, $D_2$, and $D_3$ such that
\[
3|C| \le \sum_{i=1}^3|E(D_i)| = |E(C)| + 2(2+3) = |C| +10,
\]
which implies that $|C| \le 5$, a contradiction.
\end{proof}

\begin{figure}
    \begin{subfigure}{0.4\textwidth}
        \centering
        \includegraphics[width=0.75\linewidth]{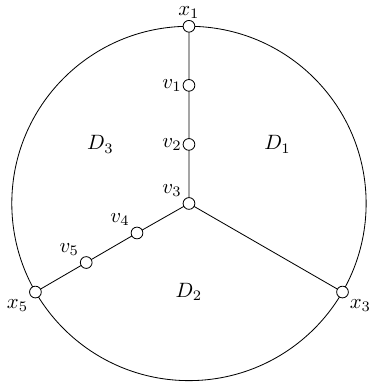}
        \caption{Cycles in Lemma \ref{lem-shortcycle}(3)}
        \label{fig:2-3(3)}
    \end{subfigure}
    \begin{subfigure}{0.4\textwidth}
        \centering
        \includegraphics[width=0.75\linewidth]{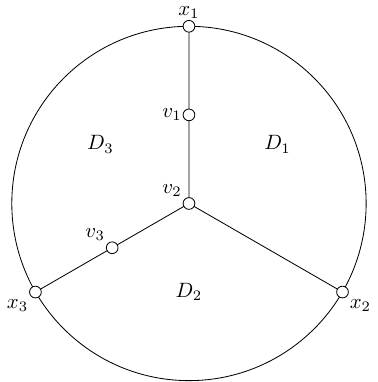}
        \caption{Cycles in Lemma \ref{lem-shortcycle}(4)}
        \label{fig:2-3(4)}
    \end{subfigure}
    \caption{} 
\end{figure}

Our proof of Theorem \ref{thm-main} is structured based upon a decomposition of the critical graphs under consideration into cycles. To this end, we define a {\it cycle partition} of a graph $G$ as a partition of $V(G)$ into  $V_0, V(C_1), \ldots, V(C_\ell)$ such that:
\begin{enumerate}
\item $C_1$ is a shortest cycle in $\underline{G}$;
\item For $2\leq i \leq \ell$, $C_{i}$ is a shortest cycle in  $\underline{G} - \bigcup_{h=1}^{i-1} V(C_h)$;
\item $\underline{G} -\bigcup_{h=1}^{\ell}V(C_h)$ is acyclic and  $V_0=V(G -\bigcup_{h=1}^{\ell}V(C_h))$.
\end{enumerate}
To simplify notation, we let $n=|V(G)|$, $n_0=|V_0|$, and we let $G_0$ be the underlying simple graph of $G[V_0]$. Note that since $C_{i}$ is a shortest cycle in  $\underline{G} - \bigcup_{h=1}^{i-1} V(C_h)$,  we may apply Lemma \ref{lem-shortcycle} to $C_i$ within this subgraph. In particular, Lemma \ref{lem-shortcycle}(1) says that any $v\in V(C_{i+1})\cup \cdots \cup V(C_{\ell})\cup V_0$ has at most one neighbor on $C_i$.

\section{The Proof of  Theorem~\ref{thm-main} for \texorpdfstring{$g\geq 6$}{ggeq6}}\label{mainpf}

Let \(g \ge 6\), and let \(G\) be a graph satisfying
\(g(G) \ge g\), \(\mu(G) \ge \lfloor g/2 \rfloor + 1\), and
\(\chi'(G) = \Delta(G) + \lceil \mu(G)/\lfloor g/2 \rfloor \rceil\).
If \(g\) is odd and $g \ge 7$, then $\lfloor g/2\rfloor=\lfloor (g-1)/2\rfloor$, so since \(g(G) \ge g \ge g-1\), the truth of the theorem for \(g\) follows immediately from the truth of the theorem for \(g-1\).
Hence, we may assume that \(g\) is even.

We will show that $G$ contains a ring graph with the same chromatic number as $G$.
We first claim that we may assume $G$ is a critical graph.
Under this assumption, it suffices to prove that $G$ itself is a ring graph.  To this end, let $H$ be a critical subgraph of $G$ such that $\chi'(H) = \chi'(G)$.  Since $H\subseteq G$, it follows that  $g(H) \ge g(G) \ge g$, $\mu(H) \le \mu(G)$, and $\Delta(H) \le \Delta(G)$. Hence,
$$\chi'(H) =\chi'(G) =\Delta(G) +\roundup{\frac{\mu(G)}{g/2}} \ge \Delta(H) + \roundup{\frac{\mu(H)}{g/2}}.$$
On the other hand, by Theorem~\ref{Thm-Steffen}, $$\chi'(H) \le \Delta(H) + \roundup{\frac{\mu(H)}{\rounddown{g(H)/2}}}.$$
Since $\Delta(H) \le \Delta(G)$ and $\roundup{\frac{\mu(H)}{\rounddown{g(H)/2}}} \le \roundup{\frac{\mu(H)}{g/2}} \le \roundup{\frac{\mu(G)}{g/2}}$,
we have that $\Delta(H) = \Delta(G)$ and $$\roundup{\frac{\mu(H)}{\rounddown{g(H)/2}}} = \roundup{\frac{\mu(G)}{g/2}}=\roundup{\frac{\mu(H)}{g/2}}.
$$
Since $\mu(G) \ge g/2+1$, it follows that $\roundup{2\mu(G)/g} \ge 2$, and so $\mu(H)\geq g/2+1$.  We further have that $\chi'(H)=\Delta(H) + \roundup{2\mu(H)/g}$. Hence, $H$ satisfies the condition of Theorem~\ref{thm-main}.  We assume $G$ is critical,  and  let $\Delta =\Delta(G)$ and $\mu=\mu(G)$.

Let $V_0, V(C_1), \ldots, V(C_{\ell})$ be a cycle partition of $G$. Since $\chi'(G) > \Delta$, $\underline{G}$ is not acyclic,  so $\ell\geq 1$. Since $|V(C_i)| \ge g$ for each $i\in [\ell]$,  $n\geq \ell g$.  By the second part of Lemma~\ref{critical},
\begin{eqnarray}\label{eq: deltaell}
\delta(G)\geq  n\mu/g +1 \geq \ell\mu+1,
\end{eqnarray}
and so
\begin{eqnarray}\label{eq: deltaUell}
\delta(\underline{G})\geq  \ell+1.
\end{eqnarray}

\begin{claim} \label{claim5}
$\delta(G_0) =1$, $\delta(\underline{G})=\ell+1$,
$n \le (\ell +1) g -1$ and $1 \le n_0\leq g-1$.  Consequently, each vertex  $v\in V_0$ with $d_{G_0}(v) = \delta(G_0) =1$ has exactly one neighbor on  $C_i$ for  each $i\in [\ell]$.
\end{claim}

\begin{proof}
We first show that $n_0\geq 1$, i.e., $V_0 \ne \emptyset$. Suppose to the contrary that $V_0 = \emptyset$.  If $\ell =1$, $G=G[V(C_1)]$ is a ring graph, and we are done. Hence, we may assume $\ell\geq 2$. Since $|C_\ell| \ge 6$, by Lemma~\ref{lem-shortcycle} (4),  there exists a vertex  $u\in V(C_\ell)$ such that $N_{C_1}(u) = \emptyset$.   By Lemma~\ref{lem-shortcycle}(1), $u$ has at most one neighbor in each $C_i$ for $i=2, \dots, \ell -1$. Since $C_\ell$ is a shortest cycle in  $G-\bigcup_{i=1}^{\ell -1}V(C_i)$,  $u$  has exactly two neighbors in $C_\ell$. Hence,  $d_{\underline{G}}(u) \le \ell$, contradicting (\ref{eq: deltaUell}).

By Lemma~\ref{lem-shortcycle}(1), each vertex in $V_0$ has at most one neighbor in $C_i$ for each $i\in\{1, \dots, \ell\}$. Hence, $d_{\underline{G}}(v) \le d_{G_0}(v) +\ell$ for each vertex $v\in V_0$.
By definition,  $G_0$ is acyclic, and so $\delta(G_0) \le 1$.
On the other hand,  $\delta(\underline{G}) \ge \ell +1$ by (\ref{eq: deltaUell}). Hence, $\delta(G_0) =1$ and $\delta(\underline{G}) = \ell +1$. Moreover, for each $v\in V_0$ if $d_{G_0}(v) =\delta(G_0) =1$, then $|N_{C_i}(v)| =1$ for each $i\in [\ell]$, which proves the second part of the claim.

Since $\delta(\underline{G})=\ell+1$, we have that
$\delta(G)\leq (\ell+1)\mu$. As we additionally have that
$\delta(G) \ge n\mu/g +1$ by \eqref{eq: deltaell}, it follows that  $n  \le  (\ell +1)g-1$. As $n_0\leq n-\ell g$, this also implies that $n_0\leq g-1$.
\end{proof}

Given a positive integer $t$, a cycle $C_h$ for some $h \in [\ell]$, and a vertex $v_0 \in V_0$, a {\it  $t$-fan $T$ from $v_0$ to $C_h$} is the union of $t$ pairwise internally disjoint paths $P_1, \dots, P_t$ in $\underline{G}[V_0\cup V(C_h)]$ from $v_0$ to $C_h$ such that for each path $P_i$, all of its vertices are in $V_0$ except for the last vertex of $P_i$, which is on the cycle $C_h$.
For such a $t$-fan $T$, we define
\(T^0 := T - V(C_h) = T \cap G_0\)
to be the subtree of $T$ contained in $G_0$. Note that   $|T^0| = |T| -t= |E(T)|-(t-1)$.

For each cycle $C_h$ with $h \in [\ell]$, since $G_0$ has no isolated vertices and every leaf of $G_0$ has a neighbor on $C_h$,
there is a $2$-fan from any of the vertices in each component of $G_0$ to $C_h$.
Moreover, if a vertex $v_0 \in V_0$ has degree at least $t$ in $\underline{G}[\,V_0 \cup V(C_h)\,]$, then there exists a $t$-fan from $v_0$ to $C_h$.

\begin{claim}\label{cla-tFan}
If there is a $t$-fan $T$ from a vertex $v_0\in V_0$ to $C_h$ for some $h\in [\ell]$, then $t\leq 3$ and
\[
|T^0| \ge (t-1)|C_h|/2 -(t-1).
\]
Consequently, $t\leq 3$. Moreover, if $t=2$, then $|T^0|\geq g/2 - 1$ and if $t=3$, then $|T^0|\geq g-2$ and $G_0$ is a tree.
\end{claim}
\begin{proof}
Let $T$ be as described in the claim, consisting of paths $P_1, \dots, P_t$ whose endpoints on $C_h$ are $x_1, x_2, \dots, x_t$, respectively. Assume without loss of generality that the endpoints $x_1,\dots,x_t$ occur in order of their indices along a fixed orientation of $C_h$. For each $i\in [t]$, let $D_i:= P_i[v_0,x_i]C_h[x_i,x_{i+1}]P_{i+1}[x_{i+1},v_0]$, where $x_{t+1} =x_1$ and $P_{t+1} =P_1$ (See Figure \ref{fig:3-2}).
\begin{figure}[h]
    \centering
    \includegraphics[width=0.3\linewidth]{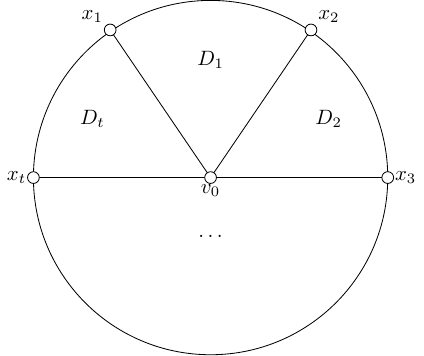}
    \caption{Cycles in Claim \ref{cla-tFan}}
    \label{fig:3-2}
\end{figure}

Note that each $D_i$ is a cycle in $\underline{G} -  \bigcup_{s=1}^{h-1}V(C_s)$,  so $|D_i|\geq |C_h|$. In the sum $\sum_{i=1}^t |E(D_i)|$, each edge in $C_h$ is counted once and each edge in $T$ is counted twice, so
 \[
 t|C_h| \le \sum_{i=1}^t|E(D_i)| = 2|E(T)| +|E(C_h)|=2(|T^0| +t-1) +|C_h|.
 \]
Hence, $|T^0| \ge (t-1)|C_h|/2 -(t-1)$.

If $t\ge 4$ then $|T^0| \ge 3|C_h|/2 -3 > |C_h| \ge g$, giving a contradiction to $n_0 \le g-1$ (Claim~\ref{claim5}). If $t=2$, then $|T^0| \ge |C_h|/2 -1 \ge g/2 -1$.
If $t=3$, then $|T^0| \ge |C_h| -2\ge g-2$. Since $n_0 \le g-1$, it follows that $|V_0\setminus V(T^0)| \le 1$. Since $G_0$ does not contain an isolated vertex, $G_0$ is a tree.
\end{proof}

We divide the remainder of the proof into a number of cases according to the values of $g$ and $\ell$.  Note that when $\ell=1$, we have that $n\leq 2g-1$ by Claim~\ref{claim5} and we divide into two subcases according to whether we have that either $n<2g-1$ or $|C_1|>g$ or we have that $n=2g-1$ and $|C_1|=g$.

Let $W=\{w\in V(G) \ : \ N_{C_1}(w)=\emptyset\}$.

\subsection{\texorpdfstring{${\mathbf{g\geq 8}}$}{ggeq8} and \texorpdfstring{${\mathbf{\ell\geq 3}}$}{lgeq3}}\label{S:bigEven}

Since $C_1$ is a shortest cycle in $\underline{G}$ and $|C_1| \ge 7$,  by Lemma~\ref{lem-shortcycle}(2), for every two adjacent vertices outside $C_1$, at least one of them is in $W$ (i.e. has no neighbors on $C_1$).  Hence, for each $2\le i \le \ell$, we have that
\[|W\cap V(C_i)| \ge \roundup{\frac{|V(C_i)|}{2}} \ge \roundup{\frac{|C_1|}{2}}\ge \frac{g}{2}.\] Since $\ell\ge 3$ this implies that $|W| \ge g$. Since each vertex in $V(G)\setminus (V(C_1)\cup W)$ has at most one neighbor on $C_1$ (by Lemma \ref{lem-shortcycle}(1)),
$|E(\underline{G}(V(C_1), V(G)\setminus V(C_1))| \le n -|C_1| - |W|$.
Hence, there exists a vertex $v_1\in V(C_1)$ such that
\begin{align*}
    d_{G}(v_1)\leq \left(\dfrac{n-|C_1|-|W|}{|C_1|}+2\right)\mu\leq  \left(\dfrac{n-2g}{g}+2\right)= \left(\dfrac{n}{g}\right)\mu,
\end{align*}
a contradiction to Lemma~\ref{critical}.

\subsection{\texorpdfstring{${\mathbf{g= 6}}$}{g=6} and \texorpdfstring{$\mathbf{\ell\geq 4}$}{lgeq4}.}

Since $g\geq 6$, by Lemma~\ref{lem-shortcycle}(4),  every three consecutive vertices of any cycle $C_j$ for $j\geq 2$ have at most two neighbors in \( C_1 \). Hence, $|W\cap C_i| \ge |C_i|/3 \ge |C_1|/3 $ for each $i\in \{2, \dots, \ell\}$. Since $\ell\geq 4$, it follows that $|W|  \ge  |C_1|$. Hence, there is a vertex $v_1\in C_1$ such that
$$d_{G}(v_1)\leq \left(\frac{n-|C_1|-|C_1|}{|C_1|}+2\right)\mu=\left(\frac{n}{|C_1|}\right)\mu\le \left(\frac{n}{g}\right)\mu,$$
a contradiction to Lemma~\ref{critical}.

\subsection{\texorpdfstring{${\mathbf{g= 6}}$}{g=6} and \texorpdfstring{$\mathbf{\ell=3}$}{l=3}.}
Let $P=v_0a_1\ldots a_tw_0$ be a longest path in $G_0$. Since $\delta(G_0)=1$ (by Claim \ref{claim5}), $v_0$ and $w_0$ are distinct leaves in $G_0$ (since $G_0$ is acyclic). By Claim \ref{claim5},  they are both adjacent to vertices on each of $C_1, C_2$, and $C_3$.

If $|P|=2$, then $|C_1|=|C_2|=|C_3|=6$; otherwise, we get a contradiction to Lemma \ref{lem-shortcycle}(2). Since $n$ is odd,   $n_0 = n -18$ is odd.  On the other hand,  since the longest path $P$ of $G_0$ contains two vertices and $G_0$ does not contain isolated vertices, each component of $G_0$ has exactly two vertices.  Hence $n_0$ is even, a contradiction.

If $|P|=3$, by Lemma \ref{lem-shortcycle}(4), $a_1$ cannot have any neighbors on $C_1, C_2$, or $C_3$, as both $v_0$ and $w_0$ have neighbors in each of these three cycles. Hence, $d_{G_0}(a_1) \ge \delta(\underline{G}) = 4$ (by Claim \ref{claim5}), so  there is a $4$-fan  from $a_1$ to $C_1$, contradicting Claim \ref{cla-tFan}.

We now consider the case $|P|\geq 4$. By Claim~\ref{claim5}, $n_0\le g-1 =5$. We claim that $G_0 = P$. Otherwise, let $x_0$ be the vertex in $V_0\setminus V(P)$. Since $G_0$ does not contain isolated vertices and $P$ is a longest path in $G_0$, we may assume $x_0\sim a_1$. Thus, $d_{G_0}(a_1) =3$. Since $\delta(\underline{G}) =4$ (by Claim \ref{claim5}), there is some $h\in [3]$ such that $N_{C_h}(a_1) \ne \emptyset$.
Hence, there is a $4$-fan from $a_1$ to $C_h$, giving a contradiction to Claim~\ref{cla-tFan}.

Since  $\delta(\underline{G}) = 4$, each of $a_1$ and $a_2$ must have neighbors on at least two of $C_1, C_2, C_3$. But then there is some cycle $C_i$ among $C_1, C_2$, and $C_3$ such that each of $v_0, a_1$ and $a_2$ have a neighbor on $C_i$, contradicting Lemma \ref{lem-shortcycle}(4).

\subsection{\texorpdfstring{${\mathbf{g\ge 8}}$}{ggeq8} and \texorpdfstring{$\mathbf{\ell= 2}$}{l=2}.} Let $P=v_0a_1\ldots a_tw_0$ be a longest path in $G_0$.  Since $G_0$ does not contain isolated vertices (by Claim \ref{claim5}), $v_0$ and $w_0$ are distinct leaves in $G_0$. Hence, there is a $2$-fan $T$ such that $T^0=P$.  By Claim~\ref{cla-tFan},   $|P| \ge g/2 -1 \ge 3$ as $g\ge 8$, so $t\geq 1$.

By Claim \ref{claim5}, $v_0$ has a neighbor on both $C_1$ and $C_2$, which in turn shows that  $a_1$ does not have any neighbors in $V(C_1)\cup V(C_2)$ (by Lemma~\ref{lem-shortcycle}(2) and as $g>7$).
By Claim~\ref{claim5},
$\delta(\underline{G})=3$. Hence, $a_1$ has at least one neighbor, say $v$, in $V_0\setminus V(P)$. Thus, $va_1P[a_1, w_0]$ is also a longest path in $G_0$, so $v$ must also be a leaf in $G_0$.
By Claim \ref{claim5},
$v$ has a neighbor on $C_1$. It follows that there is a $2$-fan $T$ from $a_1$ to $C_1$ such that $T^0=v_0a_1v$. By Claim~\ref{cla-tFan}, $3=|T^0|\ge g/2-1$, which implies that  $g=8$. Note that $n_0\leq g-1=7$ by Claim \ref{claim5}.

Since each of $v_0, w_0$, and $v$ has a neighbor on $C_1$, there is a 3-fan from $a_1$ to $C_1$. Thus, by Claim~\ref{cla-tFan}, there are at least $g-2=6$ vertices from $G_0$ in this 3-fan, so $|P|+1 \geq 6$; i.e., $|P|\geq 5$ and $a_2, a_3$ exist. Since $v_0, w_0$ are leaves in $G_0$, $|N_{C_h}(v_0)| + |N_{C_h}(w_0)| =2$ for each $h\in [2]$. If $|P|=5$, then by Lemma~\ref{lem-shortcycle}(3), for each $h\in [2]$, $\sum_{i=2}^3|N_{C_h}(a_i)| =0$.  Since $\delta(\underline{G})=3$, vertices  $a_2$ and $a_3$ must each have a neighbor in $V_0\setminus V(P)$. Since $G_0$ is acyclic, these neighbors must be outside $V(P)\cup\{v\}$ and distinct, which in turn shows that  $n_0\geq |P|+3\geq 8$, a contradiction. If $|P| =6$, then $V_0 = V(P)\cup \{v\}$ as $n_0\le 7$. By Lemma~\ref{lem-shortcycle}(3), for each $h\in [2]$, $\sum_{i=2}^4|N_{C_h}(a_i)| \le 1$. Hence, one of $a_2, a_3, a_4$ has a neighbor in $V_0\setminus (V(P)\cup \{v\})$, giving a contradiction to $V_0 = V(P)\cup \{v\}$.

\subsection{\texorpdfstring{${\mathbf{g= 6}}$}{g=6} and \texorpdfstring{$\mathbf{\ell=2}$}{l=2}.} In this case $\delta(\underline{G}) =3$ by Claim~\ref{claim5}.
Since $n$ is odd, $\underline{G}$ is not a cubic graph. Hence, there is a vertex $v\in V(G)$ with  $d_{\underline{G}}(v) = t\geq 4$.
Let $N_1=N_{G}(v)$, $N_2= N_G(N_G(v)) \setminus N_G(v)$, $X=\{v\}\cup N_1\cup N_2$ and $Y= V(G) \setminus X $. Since every two vertices in $X$ are distance at most $4$ apart, the graph $\underline{G}[X]$ induced by $X$ in $\underline{G}$ is a tree, and $N(u)\cap N(w) = \{v\} $ for any two distinct vertices $u, w\in N(v)$. Hence, $|N_2| \ge 2t$ and $|X| \ge 1+3t \ge 13$.
Since $n\leq (\ell+1)g-1=17$, this means that $|Y| \le 4$. Since $\delta(\underline{G}) =3$ by Claim~\ref{claim5}, every vertex $w\in N_2$ has at least two neighbors outside $X$. Denote by $Y_w$ a $2$-set of neighbors of $w$ in $Y$. Since $|Y| \le 4$, there are at most six subsets of $Y$ of size $2$. Thus, there are two vertices $w_1, w_2\in N_2$ such that $Y_{w_1} = Y_{w_2}$. Thus, $\underline{G}$ has a $4$-cycle, a contradiction.

\subsection{\texorpdfstring{${\mathbf{\ell= 1}}$}{l=1}, and \texorpdfstring{${\mathbf{n<2g-1}}$}{n<2g-1} or \texorpdfstring{$\mathbf{ |C_1| > g}$}{|C1|>g}}
In this case, we first claim that $G_0$ is a path,   $|N_{G_0}(C_1)|=2$, and $\sum_{v\in V_0}|N_{C_1}(v)| \le 2$.
To justify this, we consider the case that $n<2g-1$ and the case that $|C_1|>g$.
Note first that if $n< 2g-1$, then $n \le 2g -3$
as $n$ is odd,  and so $n_0 \le n -|C_1| \le g-3$. Second, if $|C_1| \ge g+1$, then $n_0 \le 2g-1-|C_1|\le |C_1| -3$. Hence in both cases, we have that $n_0\le |C_1| -3$.
If $G_0$ is disconnected, let $T_1, T_2$ be two components of $G_0$. Each of $T_1, T_2$ is a tree with at least two leaves and hence creates a 2-fan to $C_1$ (by Claim \ref{claim5}).   By Claim \ref{cla-tFan}, $n_0 \ge |T_1|+ |T_2| \ge 2(|C_1|/2 -1) =2|C_1| -2$, giving a contradiction to $n_0 \le |C_1| -3$.
Hence  $G_0$ is connected. If there is a $3$-fan $T$ from a vertex $v_0\in V_0$ to $C_1$, then by Claim~\ref{cla-tFan}, $|T^0| \ge |C_1|-2$, again contradicting $n_0\leq |C_1|-3$.  Thus, $G_0$ is connected and $\Delta(G_0) \le 2$, and so $G_0$ is a path. Consequently,  $|N_{G_0}(C_1)|=2$ and $\sum_{v\in V_0}|N_{C_1}(v)| \le 2$, as otherwise there is a $3$-fan from a vertex $v\in V_0$ to $C_1$.

By Claim \ref{cla-tFan},  $n_0\geq |C_1|/2-1$, and hence  $n\geq 3|C_1|/2 -1 \ge 3g/2-1$.
Since $\sum_{v\in V_0}|N_{C_1}(v)| \le 2$,
let $N_{C_1}(G_0) = \{x, y\}$. If $x=y$, then $\underline{G}$ has a cycle of length $n_0+1 \le |C_1|-2$, a contradiction.  Thus, $x\ne y$,  and so $\underline{G}$ is a $\Theta$-graph consisting of three internally disjoint paths from $x$ to $y$, say $P_1$, $P_2$, and $P_3$. Since $n_0\le g-3$ and $g(G) \ge g$, vertices $x$ and $y$ are not adjacent on $C_1$, and we have that $|P_i|\geq 3$ for all $i$. Since $G$ is not a bipartite graph (as $\chi'(G) > \Delta(G)$) and $n$ is odd, one of $P_1, P_2, P_3$ has an odd number of vertices and the other two have an even number of vertices; assume that $|P_1|$ and $|P_2|$ are even and that $|P_3|$ is odd.

Let $P_i = xv_{i,1}v_{i,2} \dots, v_{i, t_i} y$ for each $i\in [3]$. Then, $t_1$ and $t_2$ are even and $t_3$ is odd, and $H:=\underline{G}-\{v_{1,1}v_{1,2}, v_{2,1}v_{2,2}\}$ is a bipartite graph. Let $A, B$  be the color classes of $H$. Then, $x, y$ are in the same color class, say $A$. Moreover, it is not difficult to see that $|A| \le \rounddown{|V(G)\setminus\{v_{1,1}, v_{2,1}\}|/2} = (n-3)/2$.    We have that
$$
|E(G)|\leq |A| \Delta +|E(v_{1,1}, v_{1,2})| + |E(v_{2,1}, v_{2,2})| \leq \left(\dfrac{n-3}{2}\right)\Delta+2\mu.
$$
By Theorem \ref{Thm-CJZ} and since $G$ is critical,  $\chi'(G)=\Gamma(G) =\roundup{2|E(G)|/(n-1)}$. Hence,
$$\Delta+2 \leq \chi'(G)=\Gamma(G) \leq \left\lceil \frac{2\left(\left(\tfrac{n-3}{2}\right)\Delta+2\mu\right)}{n-1}\right\rceil=\Delta+\left\lceil\frac{4\mu-2\Delta}{n-1}\right\rceil.  $$
Consequently,  $4\mu-2\Delta> n-1$.
Since $n$ is odd, we have $4\mu-2\Delta\geq n+1$, and so
\[
\Delta\leq 2\mu-\dfrac{n+1}{2}.
\]

Directly applying  $2|E(G)| = \sum_{v\in V(G)} d_G(v) \leq n\Delta$, we get the following.
$$\Delta+ \left\lceil \frac{2\mu}{g}\right\rceil=\chi'(G)=\Gamma(G) =\roundup{\frac {2|E(G)|}{n-1}}\leq \left\lceil \frac{n\Delta}{n-1}\right\rceil=\Delta+\left\lceil\frac{\Delta}{n-1}\right\rceil.$$
Hence, $\roundup{2\mu/g} \le \roundup{\Delta/(n-1)}$, and so $\left\lceil {2\mu/g} \right\rceil -1<{\Delta/(n-1)}$. Substituting  $\Delta\leq 2\mu-\tfrac{n+1}{2} \leq g \roundup{\tfrac{2\mu}g} -\tfrac{n+1}2$, we have
\begin{equation}\label{eq: nDelta}
\left\lceil \frac{2\mu}{g} \right\rceil -1< \frac{\Delta}{n-1}\le \frac{1}{n-1} \left(g\roundup{\frac{2\mu}g} -\frac{n+1}2\right).
\end{equation}
Since $\frac{g} {n-1}\le 1$ and $\roundup{\frac{2\mu}g} \ge 2$,
it follows that
$1< \frac{1}{n-1}(2g -\frac{n+1}2)$, and so $n < (4g+1)/3$.
On the other hand, $n = n_0 +|C_1| \geq (g/2-1) +g = 3g/2 -1$. Solving inequality  $3g/2 -1 <  (4g+1)/3$,  we get $g < 8$, and so $g=6$. Then, $8 =3\cdot 6/2 -1 \le n < (4\cdot 6+1)/3 =8.33\dots $. So, $n=8$, giving a contradiction  to the fact that $n$ is odd.

\subsection{\texorpdfstring{${\mathbf{\ell= 1}}$}{l=1}, \texorpdfstring{${\mathbf{n=2g-1}}$}{n=2g-1} and \texorpdfstring{$\mathbf{ |C_1| =g}$}{|C1|=g} } In this case, $n_0=g-1$.  We first prove the following claim.

\begin{claim}\label{cla: chi'} The following two statements hold.
\begin{enumerate}
\item $\chi'(G)=\Delta+2$,  $\mu=g$, and $\delta(G)=2\mu$.
\item If there exist vertices $v, x$ and $y$ with $v\sim x$ and $v\sim y$ such that  $d_{\underline{G}}(v)\geq 3$ and $d_{\underline{G}}(x)=d_{\underline{G}}(y)=2$, then $x$ and $y$ are the only vertices of degree two in $\underline{G}$.
\end{enumerate}
\end{claim}
\begin{proof}
We prove the two statements separately.

(1) By Claim \ref{claim5},
$\delta(\underline{G}) =2$,  and so $\delta(G) \le 2\mu$. By Lemma~\ref{critical}, we have that
$$
2\mu \ge \delta(G)\ge (n-1)(\chi'(G) -\Delta -1)+2= (2g-2)\left(\roundup{\frac{2\mu}g}-1\right) +2.
$$
Hence,
\[
(g-2)\left(\roundup{\frac{2\mu}g} -1\right) + g\left(\roundup{\frac{2\mu}g} -\frac{2\mu}g\right) \le g-2.
\]
Since $\roundup{2\mu/g}\ge 2$ and $\roundup{2\mu/g} \ge 2\mu/g$, both equalities hold; i.e., $2\mu/g = \roundup{2\mu /g}=2$.  Hence,  $\chi'(G) =\Delta +\roundup{2\mu/g} =\Delta +2$ and   $\mu=g$.
By Lemma~\ref{critical}, $\delta(G) \ge n\mu/g +1=(2g-1)g/g+1=2g=2\mu$, which in turn gives  $\delta(G) =2\mu$ since $\delta(G)\leq 2\mu$.

(2) Let $v, x, y$ be as described. Since $\delta(G) = 2\mu$ (by (1)) and $d_{\underline{G}}(x) =d_{\underline{G}}(y) =2$, it follows  that $d_G(x) = d_G(y) = 2\mu$, which in turn gives $\mu(v, x)=\mu(v, y) =\mu$.  Thus, $d_G(v)\ge 2\mu+1$.

Assume for a contradiction that there are  $t\ge 3$ vertices of degree $2$ in $\underline{G}$. Then the total degree is
$2|E(G)| \le (n-t)\Delta + 2t\mu = (2g-1-t) \Delta+2tg$.
So,
$$ \chi'(G)=\Delta +2= \roundup{\frac{2|E(G)|}{2g-2}}
\leq  \roundup{\frac{(2g-1-t)\Delta+2tg}{2g-2}}  = \Delta + \roundup{\frac{\Delta-t(\Delta-2g)}{2g-2}}. $$
Consequently, $\frac{\Delta-t(\Delta-2g)}{2g-2} >1$. Since $t\ge 3$,  $\frac{\Delta-3(\Delta-2g)}{2g-2} >1$, which implies that $\Delta < 2g +1$, giving a contradiction to $d_G(v) \ge 2\mu +1 = 2g+1$.
\end{proof}

Suppose now that $g=6$, and so $n_0 =5$.
Recall that every component of $G_0$ is a tree with at least two leaves, each of which is adjacent to $C_1$ (by Claim \ref{claim5}). Since $G$ contains no $4$-fan (by Claim \ref{cla-tFan}) it follows that $\delta(G_0)\leq 3$.
Hence $G_0$ is one of the following graphs: (1) the disjoint union of two paths of lengths $2$ and $3$, respectively;
(2) a path on five vertices; and (3) a copy of $K_{1,3}$ with one edge subdivided.

Let $uvw$ be a subpath of $G_0$ such that $u$ is a leaf and suppose that in Case (1) and Case (3),  $w$ is also a leaf.  Let $x\in N_{C_1}(u)$ and let $y$ and $z$ be the two neighbors of $x$ on $C_1$.  Let $V_0\setminus \{u, v, w\} = \{a, b\}$. Then $a \sim b$ by the choice of the path $uvw$. Since $g(G) =6$ and $u\sim x$, $N(\{y, z\})\cap \{v, w\} = \emptyset$ and $N(y)\cap N(z) =\{x\}$. If $N(\{y, z\}) \supseteq \{a, b\}$, then $G$ contains a cycle with five vertices, a contradiction. We may assume that $N(y)\cap \{a, b\} = \emptyset$, which in turn implies that $N(y)\cap V(G_0) = \emptyset$.  Hence,  $d_{\underline{G}}(y) =2$.
By Claim \ref{claim5} and Lemma \ref{lem-shortcycle}(1),
all leaves in $G_0$ have degree exactly two in $\underline{G}$; in particular, $d_{\underline{G}}(u) =2$. We now find three vertices $x, u, y$ with  $d_{\underline{G}}(x) \ge 3$, $u, y\in N(x)$ and $d_{\underline{G}}(u) =d_{\underline{G}}(y)=2$.  By Claim~\ref{cla: chi'}(2), $u$ and $y$ are the only vertices of degree 2 in $\underline{G}$. However, $G_0$ must contain at least one more leaf besides $u$, which will also have degree exactly 2 in $G_0$, a contradiction.

\vskip .2in
We may now assume that $g\geq 8$, and prove the following claim.

\begin{claim}\label{cla-2g-1}
$|N_{G_0}(C_1)|\le 4$.
\end{claim}
\begin{proof} Suppose, to the contrary, that $\{v_1,\ldots, v_5\}\subseteq N_{G_0}(C_1)$ for distinct $v_1, \ldots, v_5$. By Claim~\ref{cla-tFan}, $G_0$ has at most two components, so at least three of $v_1, \dots, v_5$ are in the same component $H$ of $G_0$.
Let $F$ be a subtree of $G_0$ such that $|V(F)\cap \{v_1,\dots,v_5\}|\geq 3$ with the minimum number of vertices. Without loss of generality, suppose $\{v_1,v_2,v_3\}\subseteq V(F)$. Then all leaves of $F$ are in $\{v_1,v_2,v_3\}$; otherwise, if there is a leaf of $F$ not in $\{v_1,v_2,v_3\}$, we may remove it to obtain a smaller subtree $F'$ such that $\{v_1,v_2,v_3\}\subseteq V(F')$. It follows that $\{v_4,v_5\}\cap V(F) = \emptyset$. Hence, $\Delta(F)\leq 3$.
If $\Delta(F) = 3$, then $v_1, v_2$ and $v_3$ are the leaves of $F$, and therefore there is a $3$-fan from a vertex of $F$ to $C_1$. If $\Delta(F) = 2$, we may assume that $v_1$ and $v_3$ are the leaves of $F$, which implies that there is a $3$-fan from $v_2$ to $C_1$. Thus, in either case, there exists a $3$-fan from a vertex in $F$ to $C_1$.
By Claim~\ref{cla-tFan}, $|F| \geq g-2$ for such a fan, which yields $|G_0 - \{v_4, v_5\}| \ge |F| \ge  g-2$, contradicting the fact that  $n_0 = g-1$.
\end{proof}

Note that all leaves of $G_0$ have degree $2$ in $\underline{G}$ and
by Lemma~\ref{lem-shortcycle}(1), each vertex in $G_0$ has at most one neighbor on $C_1$. Hence,
\( |N_{C_1}(G_0)| \le |N_{G_0}(C_1)| \le 4\). The graph induced by the neighborhood $N_{C_1}(G_0)$ of $G_0$ consists of disjoint subpaths $P_1, \dots, P_s$ along the cycle $C_1$, which are separated by vertices outside of $N_{C_1}(G_0)$ along $C_1$.
Clearly, $1 \le s\le 4$.
Let $x$ be one of the ends of $P_1, \dots, P_s$. According to the definition of $P_1, \dots, P_s$, the predecessor or the successor of $x$, say $y$, is not in $N_{C_1}(G_0)$.
Hence, $d_{G_0}(y)=2$, as  $y$ only has two neighbors in $C_1$. If $x$ is adjacent to a leaf of $G_0$, then by Claim~\ref{cla-2g-1}, this leaf and $y$ are the only two vertices of degree $2$ in $\underline{G}$, a contradiction as $G_0$ has another leaf. Hence, the ends of $P_1, \dots, P_s$ are not adjacent to any leaves of $G_0$. Since $G_0$ has at least two leaves, there are at least two internal vertices in these paths $P_1, \dots,P_s$, and hence some path has at least three vertices. Consequently, if $s\ge 2$ then these paths together contain at least three ends, which in turn gives $|N_{C_1}(G_0)| \ge 5$, a contradiction. Thus, $s=1$.
If $G_0$ has three leaves, then $|N_{G_0}(C_1)|\geq 5$ as each leaf of $G_0$ has a neighbor on $C_1$, a contradiction. Hence, $G_0$ has exactly two leaves, and so $G_0$ is a path.
Since  $N_{G_0}(P_1)$ contains the two leaves of $G_0$ and two neighbors of $x_1, x_t$, we have $|N_{G_0}(P_1)| =4$. The four vertices in $N_{G_0}(P_1)$ divide the path $G_0$ into three edge disjoint subpaths. One of these subpaths, denoted by $Q$, has at most $|E(G_0)|/3 =(g-2)/3$ edges, which in turn gives a cycle $D$ in $\underline{G}$ with $|E(D)| \le (g-2)/3 +2 +|E(P_1)| =(g-2)/3 +5$ edges. Solving the inequality $(g-2)/3 +5 \ge g$, we find that $g\le 6.5$, a contradiction as $g\ge 8$.

\section{The proof of Theorem~\ref{thm-main} for \texorpdfstring{$g=5$}{g=5}}\label{oddipf}

Let $G$ be a graph with $g(G) \ge 5$, $\mu(G) \ge 3$ and $\chi'(G) =\Delta(G) + \roundup{\mu(G)/2}$.
Following the same argument as in the case $g\ge 6$ and $g$ is even at the beginning of Section \ref{mainpf}, we may assume that $G$ is critical, and show that $G$ itself is a ring graph. For completeness, we give a proof below. Let $H$ be a critical subgraph of $G$ with $\chi'(H) =\chi'(G)$.  Since $H\subseteq G$, we have $g(H) \ge g(G) \ge 5$, $\mu(H) \le \mu(G)$, and $\Delta(H) \le \Delta(G)$.
By Theorem~\ref{Thm-Steffen},
$$
\chi'(H) \le \Delta(H) + \roundup{\frac{\mu(H)}{\rounddown{g(H)/2}}}.
$$
On the other hand, since $\chi'(H) =\chi'(G)$, we have
$$\chi'(H)  =\Delta(G) +\roundup{\frac{\mu(G)}{2}} \ge \Delta(H) + \roundup{\frac{\mu(H)}{2}}.
$$
The combination of the above two inequalities gives   $$
\chi'(H)=\Delta(H) + \roundup{\frac{\mu(H)}2} = \Delta(G) + \roundup{\frac{\mu(G)}{\rounddown{g(G)/2}}}.
$$
Since $\chi'(H) = \chi'(G) \ge \Delta(H) +2$, we have $\mu(H) \ge \rounddown{g(G)/2} +1\ge 3 =\rounddown{5/2} +1$. Therefore, $H$ satisfies the conditions of Theorem~1.2, and thus we may consider  $H$ instead of $G$.

Let $\Delta =\Delta(G)$, and $\mu=\mu(G)$. By Theorem \ref{Thm-Steffen},  $\chi'(G)\le \Delta+\roundup{\mu/\rounddown{g(G)/2}}$. Since $\chi'(G) = \Delta + \roundup{\mu/2}$, it follows that
\begin{equation}\label{eq:inequality0}
\roundup{\frac{\mu}{2}}\le \roundup{\frac{\mu}{\rounddown{ g(G)/2}}}.
\end{equation}

If $g(G) \ge 6$, then $G$ satisfies the condition of Theorem~\ref{thm-main} with $g\ge 6$ as $\chi' =\Delta +\roundup{\mu/\rounddown{g(G)/2}} \ge \Delta +2$, which is covered in the previous case. For this reason, we may assume $g(G) =5$.

Suppose, for a contradiction, that $G$ is not a ring graph.
Since $\chi'(G) > \Delta$, $\underline{G}$ contains a cycle. Since $G$ is not a ring graph, $n\ge g(G) +1=6$. By Theorem~\ref{critical}, $n$ is odd, so $n \ge 7$.

Let $V_0, V(C_1), \ldots, V(C_\ell)$ be a cycle partition of $G$. Since $C_1$ is the shortest cycle of $\underline{G}$, by Lemma \ref{lem-shortcycle}(1),  every vertex in $V(G)\setminus V(C_1)$ has at most one neighbor on $C_1$. So there is a vertex $v_1\in V(C_1)$ such that
\begin{align}\label{side1}
d_{\underline{G}}(v_1)\le\left(\frac{n-|C_1|}{|C_1|}+2\right)= \frac{n+5}{5} \ \mbox{ and }   d_{G}(v_1) \le \left(\frac{n+5}{5}\right)\mu.
\end{align}
On the other hand, by Lemma \ref{critical},
\begin{align}\label{side2}
    d_{G}(v_1) \ge (n-1)(\chi'-1 -\Delta) +2 \ge (n-1)\left(\roundup{\frac{\mu}{2}}-1\right) +2.
\end{align}
Together, \eqref{side1} and \eqref{side2} imply that
\begin{align}\label{eq: nbound}
 (n-1) \left({\roundup{\frac{\mu}{{2}}}}-1\right) + 2\leq \left(\frac{n+5}{5}\right)\mu.
\end{align}
When $\mu=3$ this implies that $n\leq 10$. When $\mu\geq 4$, by dropping the ceiling on the left-hand-side of \eqref{eq: nbound} we get $n\leq 5+20/(3\mu-10)$, and hence, $n \le 15$ if $\mu=4$ and $n \le 9$ if $\mu\ge 5$.

Suppose first that $\mu \neq 4$. Then, $n\in\{7, 9\}$ since $n$ is odd and  $7\leq n\leq 10$. Consequently $\ell =1$ since $g(G) =5$.  Since $G_0$ is acyclic, $\delta(G_0) \le 1$.  By Lemma~\ref{lem-shortcycle}(1), every vertex in $V_0$ has at most one neighbor on $C_1$. Hence, $\delta(\underline{G}) \le 2$,  and so  $\delta(G) \le 2\mu$.  Applying Lemma \ref{critical}, we find that $(n-1) (\roundup{\mu/2}-1) + 2\leq \delta(G) \le 2\mu$. When $\mu=3$ this implies that $n\leq 5$, a contradiction. Hence we may assume that $\mu\geq 5$, but this implies that $n \le 5 + 4/(\mu - 2)\leq 6$, a contradiction.

We  now assume that $\mu=4$. Since $n$ is odd and $7\leq n\leq 15$, we have that $n\in\{7, 9, 11, 13, 15\}$.
As $\mu=4$, we have by \eqref{side1} and \eqref{side2} that there is a vertex $v_1\in V(C_1)$ such that
\[
d_{\underline{G}} (v_1)\leq \frac{n+5}5\ \mbox{ and } \ d_G(v_1)\geq n+1.
\]
The second inequality implies that $d_{\underline{G}}(v_1)\geq \roundup{(n+1)/4}$. Hence,  $(n+5)/5\geq \roundup{(n+1)/4}$. If $n=9$, we have that $14/5 \ge \roundup{10/4}=3$, a contradiction. If $n=13$, then we get that $18/5 \le \roundup{14/4}=4$, a contradiction. Therefore, $n\in \{ 7, 11, 15\}$.

By Lemma \ref{critical},  we have $\delta(G)\geq n\mu/g+1=4n/5+1$, and so $\delta(\underline{G})\geq \roundup{n/5+1/4}$.
Hence,  $\delta(\underline{G})\geq 4$ when $n=15$, $\delta(\underline{G})\geq 3$ when $n=11$, and $\delta(\underline{G})\geq 2$ when $n=7$.

Suppose that $n=15$. Let $v\in V(G)$. Since $g(G) =5$, for any two distinct vertices $w_1, w_2\in N(v)$, we have $N(w_1)\cap N(w_2) =\{v\}$.  As $\delta(\underline{G})\geq 4$,
\[
n \ge |\{v\}| + |N_{\underline{G}}(v)| +\sum_{w\in N(v)} |N_{\underline{G}}(w)\setminus\{v\} | \ge 1+4 +4\cdot 3=17,
\]
a contradiction.

Now suppose that $n=7$. Then the cycle partition of $G$ consists of $V_0$ and $V(C_1)$ with $\ell=1$ and $n_0=2$. Since $\delta(\underline{G})\geq 2$ and a vertex in $V_0$ can be adjacent to at most one vertex on $C_1$ by Lemma \ref{lem-shortcycle}(1),  the two vertices in $V_0$ must be adjacent. Let $C_1=w_0w_1 \dots w_4w_0$ and $V(G_0) = \{x,y\}$. Since $g(G)=5$ we may assume without loss of generality that $x\sim w_0$ and $y\sim w_2$. Then $\underline{G}-\{w_0, w_2\}$ is a disjoint union of two edges $xy$ and $w_3w_4$ and an isolated vertex $w_1$. Hence, $E(G)$ can be partitioned into four subsets: edges incident with $w_0$, edges incident with $w_2$, edges between $x$ and $y$, and edges between $w_3$ and $w_4$. Consequently, as there are at most $\Delta$ edges incident with any single vertex and at most $\mu$ edges between any pair of vertices, $|E(G)| \le 2\Delta + 2\mu = 2\Delta +8$.
As $G$ is critical, $\Gamma(G) = \roundup{2|E(G)|/(n-1)}=\roundup{|E(G)|/3}$. So, the following holds.
\[
\Delta + 2 \leq \chi'(G) =\Gamma(G) \le   \roundup{\frac{2\Delta +8}{3}} = \Delta +\roundup{\frac{8 -\Delta}3}.
\]
We conclude that $8 -\Delta \ge 4$; i.e., $\Delta \leq 4$. Since $\mu=4$ and $\delta(\underline{G}) \ge 2$, this is a contradiction.

Finally, suppose that $n=11$. In this case, we have $\delta(\underline{G})\geq 3$ and  $\ell\leq 2$ as $g(G) =5$. Since $G_0$ is acyclic, $\delta(G_0) \le 1$. Since $C_1$ is a shortest cycle in $\underline{G}$, each vertex in $G-V(C_1)$ has at most one neighbor on $C_1$ by Lemma \ref{lem-shortcycle}(1). If $\ell =1$, we have $\delta(\underline{G}) \le 2$, giving a contradiction to $\delta(\underline{G})\geq 3$. Hence,  $\ell=2$. If $V_0\neq \emptyset$, then $|C_1|=|C_2|=5$, and the vertex in $V_0$ has at most one neighbor in each of $C_1, C_2$ (by Lemma \ref{lem-shortcycle}(1)), so it has at most two neighbors in $G$, contradicting $\delta(\underline{G})\geq 3$. Thus it must be the case that $V_0=\emptyset$ with $|C_1|=5$ and $|C_2|=6$. Assume that $V(C_1)=\{x_1,\dots,x_5\}$ and $V(C_2)=\{y_1,\dots,y_6\}$. Each vertex on $C_2$ is adjacent to at least one vertex on $C_1$, as $\delta(\underline{G})\geq 3$.
Since $g(G)=5$, we can assume without loss of generality that $y_1\sim x_1$ and $y_2\sim x_3$. Moreover, since there is no cycle of length less than five, we must have that $y_3\sim x_5$, $y_4\sim x_2$ and $y_5\sim x_4$. As $x_1\sim y_1$, it follows that $N(y_6)\cap \{x_1, x_2, x_5\} = \emptyset$.
Furthermore, since $y_5\sim x_4$, we also have that $N(y_6)\cap \{x_3, x_4\} = \emptyset$.
Hence, $y_6$ does not have any neighbors on $C_1$, which implies that $d_{\underline{G}}(y_6) \le 2$, a contradiction.

\bibliographystyle{plain}
\bibliography{References.bib}

\end{document}